\documentclass [12pt]{amsart}
\usepackage[utf8]{inputenc}
\usepackage[OT1]{fontenc}
\pdfoutput=1
\usepackage{amsmath,amssymb,amsfonts}
\usepackage{mathtools}%
\usepackage{dsfont}
\usepackage[english]{babel}%
\usepackage{comment}%
\usepackage[unicode]{hyperref}%
\usepackage{bbm}%
\usepackage{bm}
\usepackage{mathrsfs}%

\usepackage{tikz,graphicx,color}
\usepackage{tikz-cd}%
\usepackage{tikz-3dplot}%
\usetikzlibrary{calc}%
\usetikzlibrary{arrows}%
\usetikzlibrary{shapes}%
\usetikzlibrary{patterns}%
\usetikzlibrary{positioning}%
\usetikzlibrary{arrows.meta}
\usetikzlibrary{decorations.markings}
\usetikzlibrary{knots}

\usepackage{epstopdf}%

\usepackage[arrow]{xy}%
\usepackage{diagbox}%
\usepackage[normalem]{ulem}
\usepackage{subfig}%
\usepackage{arcs}%
\usepackage{xcolor}%
\usepackage{xspace}

\usepackage[subtle,tracking=normal,mathdisplays=tight]{savetrees}
\usepackage[margin=1.0in]{geometry}%

\usepackage{enumitem}
\usepackage{comment}
\usepackage{letltxmacro}
\usepackage{etoolbox}

\newtheorem{theorem}{Theorem}[section]

\newtheorem{lemma}[theorem]{Lemma}

\newtheorem{proposition}[theorem]{Proposition}
\newtheorem{corollary}[theorem]{Corollary}
\theoremstyle{definition}

\theoremstyle{definition}
\newtheorem{definition}[theorem]{Definition}

\theoremstyle{definition}

\theoremstyle{definition}

\theoremstyle{definition}

\newcommand{\C}{\mathbb{C}}

\newcommand{\Bk}{\mathbb{K}}
\newcommand{\Z}{\mathbb{Z}}

\newcommand{\vp}{\varphi}
\newcommand{\X}{\mathcal{X}}
\newcommand{\A}{\mathcal{A}}
\newcommand{\pt}{p}
\newcommand{\pol}{q}
\newcommand{\pun}{\pi}
\newcommand{\puns}{S}
\newcommand{\bn}{b}
\newcommand{\bns}{\mathcal{B}}
\newcommand{\claut}{\vp}
\newcommand{\taginv}{T^{\bowtie}_\pun}

\newcommand{\surf}{\Sigma}

\newcommand{\sgclaut}{\claut_\pt^{\pm}}
\newcommand{\reclaut}{\claut_{\pun,\zeta}}

\tikzset{
    node/.style = {circle, draw, fill=white, inner sep=0pt, minimum size=6pt, outer sep = 2pt},
    arrow/.style = {-{Stealth[length=5pt, width=4pt]}}
    amidarrow/.style  args={#1,#2}={postaction={decorate, decoration={markings, mark=at position #2 with {\arrow[scale=1.2]{Stealth[#1]}}}}}
}
\tikzset{square/.style={regular polygon, regular polygon sides=4, draw, inner sep=0,outer sep = 2pt,minimum size=8pt}}

\DeclareMathOperator{\Aut}{Aut}

\title[On the Mysterious Points conjecture]{Locally Acyclic Surface Type and Finite Type  Cluster Algebras Have No Mysterious Points}
\author{Matthew J. Tyler }
\date{August 2026}

\begin{document}

\begin{abstract}
     Cluster varieties contain the union of cluster tori and the points not in this union are called \emph{deep points}, and the locus of these points is called the \emph{deep locus}. In \cite{CGSS}, a description of this locus is conjectured for locally acyclic cluster algebras, in particular, stating that this should be the stabilizer locus of the cluster automorphism group. 
    We resolve this conjecture for the case of cluster algebras arising from surfaces, introduced in \cite{FST1}, and the remaining finite type cases as categorized in \cite{FZ2}. In particular, we show 
        locally acyclic surface type cluster varieties have no deep points not contained in the stabilizer locus, and finite type cluster algebras also have no deep  points not contained in the stabilizer locus, validating the mysterious points conjecture in these cases.
\end{abstract}
\maketitle
\section{Introduction}
\par In recent years, the varieties arising from cluster algebras have been the topic of much study \cite{fock2006moduli, fock2009cluster,Muller2013LocallyAcyclic,Leclerc2016-mn, gross2018canonical,Serhiyenko2019-pz,LS,Casals2024-nv,CGSS,BM,Galashin2026-rn}.  The cluster variety associated to a cluster algebra $\A$ in the sense of \cite{FZ1} is the affine scheme $\X= \operatorname{Spec}(\A)$. The generators of $\A$ can be grouped into distinct finite collections referred to as seeds with $n$ generators in each seed (referred to as the rank of the cluster algebra). These generators are called the cluster variables. The Laurent phenomenon \cite{FZ1} guarantees that each element of the algebra can be expressed as a Laurent polynomial in the cluster variables of any seed $\vec x$, giving, for each seed, an open embedding
$$(\C^\times)^n\hookrightarrow \X.$$
This is called the \textit{cluster torus} associated to $\vec x$. The points of the scheme $\X$ which are not in the image of any cluster torus are called \textit{deep points}, and the collection of deep points is referred to as the \textit{deep locus} of $\X$.

\par In \cite{CGSS}, the authors put forth the conjecture that for locally acyclic cluster algebras\cite{Muller2013LocallyAcyclic}, this deep locus is identically the stabilizer locus of the \textit{cluster automorphism group}\cite{GSV,LS,CGSS}. Sometimes called the \textit{cluster dilation group}\cite{NS}, this is the group which acts on a cluster algebra by rescaling all of the cluster variables. Thus, it can be seen that if a point in the variety is stabilized by a non-identity element of this group, that point must be deep. The conjecture states that for locally acyclic cluster algebras, all deep points should have a non-trivial stabilizer under this action. Very recently, in \cite{NS} this conjecture has been shown to be false in general; however, it remains true in several cases \cite{CGSS,NS}. We expand these known cases to cluster algebras arising from surfaces, also known as surface type cluster algebras \cite{FST1}, and complete the work in \cite{CGSS} to all finite type cluster algebras\cite{FZ2}. In this paper, we show the following two results.
\begin{theorem}
    Surface type cluster algebras which are locally acyclic have no mysterious points.\label{thm}
\end{theorem}
In particular, if the marked surface includes into a disk or has at least two marked points in each connected component, the associated cluster algebra is locally acyclic \cite{Muller2013LocallyAcyclic}. It was shown in \cite{CGSS} that type A, D, and E cluster algebras have no mysterious points. We extend that with the following.
\begin{theorem}
    Cluster algebras of finite type have no mysterious points.\label{thm2}
\end{theorem}

\par The paper is organized as follows. In Section \ref{Sec:bkg}, we provide the background of the objects and proof techniques used in the paper. In Section \ref{sec:pf}, we discuss the case of a surface without punctures. In Section \ref{sec:ps} we discuss the case of a surface with punctures, proving Theorem \ref{thm}. Finally, in Section \ref{sec:p2}, we discuss the case of Rank 2 cluster algebras, and cluster algebras of types B,C,F, and G, thus completing Theorem \ref{thm2}.

\section{Background}
\label{Sec:bkg}
\subsection{Cluster Algebras}
Cluster algebras were introduced by Fomin and Zelevinsky in \cite{FZ1}. We will use the definition of a cluster algebra associated with a skew-symmetrizable matrix.
\begin{definition}
    A \emph{quiver} $Q$ is a directed graph $(V,E)$ with no directed 2-cycles, and each vertex $v\in V$ is labeled either mutable or frozen. By convention, we denote the first $n$ vertices as the mutable vertices, and the last $m$ as the frozen vertices. We also forbid edges between frozen vertices. Let $B_Q$ denote the adjacency matrix, and $\tilde{B}_Q$ be the matrix with the first $n$ columns of $B_Q$. We call $\tilde{B}_Q$ the \emph{exchange matrix}.
\end{definition}
Note that $B_Q$ is a well-defined skew-symmetric matrix and recovers the quiver as we forbade directed 2-cycles. Furthermore, as we have forbidden arrows between frozen vertices, $\tilde{B}_Q$ also recovers the quiver. As we consider the case of finite quivers, we often denote $V=[n+m]$. To quivers, we associate a mutation rule to each vertex.
\begin{definition}
    A \emph{mutation at $k\in [n]$} of $Q$ is a quiver $Q'$ obtained from $Q$ by the following steps.
    \begin{enumerate}
        \item Initialize $Q'$ as $Q$
        \item For each path $i\to k\to j$ in $Q$, add an arrow $i\to j$.
        \item Reverse the direction of each arrow incident to $k$.
        \item Prune directed 2-cycles and arrows between frozen vertices.
    \end{enumerate}
\end{definition}
We can define mutation for a general skew-symmetrizable matrix as follows.
\begin{definition}
Given $\tilde B = (b_{ij})$ an $(n+m)\times n$ extended skew-symmetrizable integer matrix, and $k\in[1,n]$, the matrix mutation $\tilde B'=\mu_k(\tilde B)$ is given by
$$b'_{ij}=\begin{cases}
    -b_{ij}&\text{if $i=k$ or $j=k$}
    \\
    b_{ij}+b_{ik}b_{kj}&\text{if $b_{ik}>0$ and $b_{kj}>0$}
    \\
    b_{ij}-b_{ik}b_{kj}&\text{if $b_{ik}<0$ and $b_{kj}<0$}
    \\
    b_{ij}&\text{otherwise.}
\end{cases}$$
\end{definition}
It can be verified that mutation at a vertex $k$ is an involution, and that matrix mutation at position $k$ is also an involution. We say $Q$ and $Q'$ are mutation equivalent if there is a finite sequence of mutations that takes $Q$ to $Q'$. Likewise, two skew-symmetrizable matrices $\tilde B$ and $\tilde B'$ are mutation equivalent if there exists a finite sequence of matrix mutations taking $\tilde B$ to $\tilde B'$.
\\
We can now note that the set of mutation equivalent quivers to $Q$ (or skew-symmetrizable matrices mutation equivalent to $\tilde B$) can be drawn in an $n$-regular tree $\mathbb{T}_n$. To each node, we record the matrix $\tilde B_Q$, and we assign a vector of independent variables $\vec x = (x_1,...,x_{n+m})$. We call $(\vec x,\tilde B_Q)$ a cluster seed, and $\vec x$ the cluster variables associated to that cluster. We also declare that if $Q$ and $Q'$ are connected via mutation at $k$ (or $\tilde B$ and $\tilde B'$), with $\vec x$ and $\vec x'$ as the associated cluster variables, we impose that for all $i\ne k$, $x_i=x_i'$ and the cluster relation:
$$x_k\cdot x_k'=\prod_{b_{ik}>0}x_i^{b_{ik}}+\prod_{-b_{ik}<0}x_i^{b_{ik}}$$
This puts all of the cluster variables in any mutation-equivalent seed as elements of $\C(x_1,...,x_{n+m})$ where the $\vec x$ here is the initial seed. We declare the cluster algebra $\A_Q$ or $\A(Q)$ (likewise $\A_B$) to be the sub algebra of $\C(x_1,...,x_{n+m})$ generated by all cluster variables in all seeds mutation equivalent to $Q$ (or matrix mutation equivalent to $\tilde B$) along with $x_{n+1}^{-1},...,x_{n+m}^{-1}$ (the inverses of the frozen variables). 


\par We also define the \emph{cluster variety} $\X_Q$ as $\operatorname{Spec}(\A_Q)$, the closed points of which are well known to be in correspondence with elements of $\operatorname{Hom}_{\text{Ring}}(\A_Q,\C)$. The main object of study for this paper will be the so-called deep points of this variety.
\begin{definition}
    A point $\pt\in\X_Q$ is called \emph{deep} if, for all clusters $(\vec x, \tilde B_Q)$, $\prod_i \pt(x_i)=0$, i.e. $\pt$ vanishes on at least one cluster variable in each cluster. \label{def:deep}
\end{definition}
The deep points are exactly the points that are not in the cluster torus (the torus $(\C^\times)^{n+m}$ in $\X_Q$ given by $x_i\mapsto\zeta_i$ for some cluster $(\vec x, B)$) for any cluster, and thus the deep locus (the set of all deep points) is exactly those points in the variety outside the union of cluster tori.

\begin{definition}
    A cluster variety is called \emph{locally acyclic} if the variety is covered by finitely many acyclic cluster charts.
\end{definition}
While the details of this definition are not important for this discussion, in the locally acyclic case we have the following result.
\begin{proposition}
    [\cite{Muller2013LocallyAcyclic}, Lemma 3.4] If $\A$ is a locally acyclic cluster algebra, and $\A'$ is a cluster algebra obtained by freezing some vertex $v\in Q$, then $\A'=\A[x_{v}^{-1}]$.
\end{proposition}
This proposition allows us to identify charts in $\X_Q$ by freezings of vertices in $Q$.
\subsection{Cluster Algebras from Surfaces}
\par
Let $\surf$ be an orientable, triangulable surface with boundary $\partial\surf$ together with a collection $\puns$ of punctures $\pun$ and a collection $\bns$ of marked points $\bn$ along the boundary, $\partial \Sigma$. We call $(\surf,\puns,\bns)$ a \emph{marked surface}. We will assume $\surf$ is connected, but we will see in Corollary \ref{cor:multi} that this assumption is not necessary.
\par Like in \cite{Muller2013LocallyAcyclic}, we will exclude some small examples. We will assume $\puns\cup\bns$ is non-empty, each component of $\partial\surf$ contains a point $\bn\in\bns$, and $(\surf,\puns,\bns)$ is not a sphere with at most three punctures, a disc with $|\bns|=1$, $|\puns|\le1$, or a disc with $|\bns|=2$ and $|\puns|=0$.

\par We call a curve $\gamma:[0,1]\hookrightarrow\surf$ a \emph{tagged arc} if $\partial\gamma\subset\puns\cup\bns$, each endpoint of $\gamma$ in $\puns$ is denoted either \emph{plain} or \emph{notched} (the endpoints of $\gamma$ in $\bns$ are taken to be plain by convention), $\gamma$ is a proper inclusion away from the endpoints, if $\partial\gamma$ is one point, then the taggings at each end agree, the interior of $\gamma$ does not intersect $\puns\cup\bns$, and $\gamma$ does not cut out either a disc with one boundary marked point and at most one puncture or a disc with two boundary marked points and no punctures.

\par We consider tagged arcs up to endpoint-fixed homotopy, where fixing the endpoints also means fixing the taggings at the endpoints.

\begin{definition}
    Two arcs $\gamma_1, \gamma_2$ are called \emph{compatible} if one of the following.
    \begin{enumerate}
        \item $\gamma_1$ does not intersect $\gamma_2$.
        \item $\gamma_1$ and $\gamma_2$ intersect at one point, and the taggings of $\gamma_1$ and $\gamma_2$ match at that point.
        \item $\gamma_1$ and $\gamma_2$ are homotopic via an endpoint-fixing homotopy, and the tagging agrees at one endpoint and disagrees at the other.
    \end{enumerate}   
    
    A collection of arcs is called \emph{mutually compatible} if every two arcs in the collection are compatible.
\end{definition}

\begin{definition}
    A \emph{tagged triangulation} $\Delta$ of $(\surf, \puns,\bns)$ is a maximal collection of mutually compatible arcs.
\end{definition}
\begin{definition}
     Relative to a triangulation $\Delta$, we construct the \emph{associated quiver} $Q_\Delta$ as follows.
\begin{itemize}
    \item The mutable vertices in $Q_\Delta$ correspond to tagged arcs $\gamma\in \Delta$.
    \item The frozen vertices in $Q_\Delta$ correspond to the connected components $f$ of $\partial \surf\smallsetminus\bns$.
    \item For each marked point $\pun\in\puns\cup\bns$ as long as more than two curves in $\Delta\cup\{f\}$, for $\gamma_1,\gamma_2$ curves incident to $\pun$, we add a directed arrow $\gamma_1\to\gamma_2$ if $\gamma_2$ follows $\gamma_1$ in cyclic order around $\pun$ (determined by the orientation of $\surf$). For this step, if $\gamma_a$ and $\gamma_b$ are homotopic but differ in taggings, we treat them as the same curve for the purpose of this step.
\end{itemize}
\label{def:quiv}
\end{definition}

\begin{lemma}
    [\cite{FST1}, Proposition 4.10, \cite{Muller2013LocallyAcyclic}, Lemma 9.8] If $\Delta$ and $\Delta'$ are two tagged triangulations of $(\surf,\puns,\bns)$, then there is a canonical isomorphism of cluster algebras
    $$\A(Q_\Delta)\simeq\A(Q_{\Delta'}).$$
\end{lemma}
We define the \emph{cluster algebra} associated to $(\surf,\puns,\bns)$ by
$$\A_{\surf,\puns,\bns}:=\A(Q_{\Delta}),$$
where $\Delta$ is any triangulation of $(\surf,\puns,\bns)$. The algebra is well defined by the above lemma. Going forward, we will denote the triple $(\surf,\puns,\bns)$ by $\surf$, and thus this algebra will be denoted $\A_\surf$.
\begin{lemma}
    [\cite{FST1} Theorems 5.6 and 7.11, \cite{FominShapiroThurston2013SurfacesII} Theorem 6.1] If $(\surf,\puns,\bns)$ is as above, then the cluster variables in $\A_\surf$ are exactly the tagged arcs of $(\surf,\puns,\bns)$. (If $|\puns|=0$, or $|\puns|=1$ and $\partial\surf$ is empty, we have the additional restriction that all arcs must be plain arcs).\label{lem:clvars}
\end{lemma}

\begin{definition}
    The \emph{cluster variety} associated to a marked surface $\surf$ is  $\X_\surf:=\operatorname{Spec}(\A_\surf)$.
\end{definition}

\begin{definition}
    Restating Definition \ref{def:deep} for surface cluster algebras, \label{deep} A point $\pt\in\X_\surf$ is called deep if for any triangulation $\Delta$ of $\surf$ there exists some arc $\gamma\in \Delta$ with $\pt(\gamma)=0$ (Due to Lemma \ref{lem:clvars}, the functions on $\A_\surf$ are also functions on tagged arcs in $\Delta$, so we take the notation $\gamma=x_\gamma\in \A_\surf$).
\end{definition}

\begin{definition}
    We call a puncture $\pun\in\surf$ $\pt$-\textbf{essential} if for all tagged arcs $\gamma$ on $\surf$ with $\pun\in\partial \gamma$, $\pt(\gamma)=0$.
\end{definition}
Importantly, if $\pun\in\surf$ is not $\pt$-essential, then there exists some (possibly notched) arc $\gamma$ with $\pt(\gamma)\ne 0$.
\subsection{Cluster Automorphisms}
One of the major tools for this paper will be the categorization of when the torus action on a cluster extends to a morphism of the cluster algebra. \cite{LS} denoted such automorphisms as Cluster Automorphisms \footnote{not to be confused with another morphism with the same name defined in \cite{ASS12}}.

\begin{definition}
    Given a cluster algebra $\A$, a \emph{Cluster Automorphism} is an algebra morphism $\claut:\A\to\A$ such that for each cluster seed $(\vec x, B)$, for each $x_i\in \vec x$, there is some $\zeta\in\C^\times$ such that $\claut(x_i)=\zeta x_i$. I.e. a map is a cluster automorphism if it is an algebra automorphism that rescales each cluster variable.
\end{definition}
We will use the following tool to prove a map is a cluster automorphism, namely, we will start with a toric rescaling on a seed and ask when it extends to a cluster automorphism on the whole cluster algebra.

\begin{proposition}
    Given $\zeta\in\C^\times$, and $\vec w\in \Z^
    {n+m}$, and a cluster seed $(\vec x,B)$ for a cluster algebra $\A$, the toric rescaling $\claut: x_i\mapsto \zeta^{w_i}\cdot x_i$ extends uniquely to a cluster automorphism on $\A$ if 
    \begin{equation}
        B^Tw=0\label{eqn:Bwzero}
    \end{equation}
    Moreover, if $\zeta$ is a $p$th root of unity, $\claut: x_i\mapsto \zeta^{w_i}\cdot x_i$ extends uniquely to a cluster automorphism on $\A$ if
    \begin{equation}
        B^Tw\equiv0\text{ mod }p
        \label{eqn:Bwp}
    \end{equation}
    \label{prop: claut_extension}
\end{proposition}
\begin{proof}
    This statement follows directly from the characterization of cluster automorphisms, Proposition 5.1 in \cite{LS} and Lemma 2.3 in \cite{GSV}.
\end{proof}

We will use the following proposition to ensure that we do not have to worry much about the frozen variables in the finite type case. The frozen variables we will take in the surface case will be the boundary arcs.

\begin{definition}
    An exchange matrix $B$ is really full rank if its row vectors generate the standard $\Z^n$ lattice.
\end{definition}
\begin{proposition}\cite[Corollary 3.19]{CGSS}
    If $\A_B$ is a really full rank cluster algebra over $B$ with no mysterious points in its cluster variety, then for every matrix $B'$ with an identical mutable part to $B$, $\X_{B'}$ also has no mysterious points.\label{prop:RFR}
\end{proposition}

\section{The puncture-free case}
\label{sec:pf}
The goal of this section will be to prove the following.
\begin{proposition}
The cluster algebras associated with unpunctured surfaces have no mysterious points. \label{prop: unpunc}
\end{proposition}
To do this we will use the characterization of deep points of unpunctured surfaces due to Beyer and Muller.
\begin{lemma}
    (\cite{BM}, Lemma 7.1) If $\surf$ is a marked surface without punctures, and $\pt\in\X_\surf$ is a deep point, then, for every triangulation, for each triangle $\Delta=(\gamma_1,\gamma_2,\gamma_3)$ in that triangulation, $\pt$ kills exactly one or three edges of that triangle, i.e. $\#\{i|\pt(\gamma_i)=0\}=1$ or $3$.
    \label{lem: 1-or-3}
\end{lemma}
With this lemma, we can directly prove proposition \ref{prop: unpunc} by constructing what we call the \emph{sign-reversing cluster automorphism}.
\begin{definition}
    Given an unpunctured surface $\surf$ and a deep point $\pt\in\X_\surf$, let the deep sign reversing cluster automorphism $\sgclaut$ be the map as follows. For each arc $\gamma$ with $\pt(\gamma)=0$, $\sgclaut(\gamma)=-\gamma$, and for each arc $\gamma$ with $\pt(\gamma)\ne0$, $\sgclaut(\gamma)=\gamma$. 
\end{definition}
To prove Proposition \ref{prop: unpunc}, it suffices to show the following.
\begin{lemma}
    $\sgclaut$ is a cluster automorphism.
\end{lemma}
\begin{proof}
    By \ref{prop: claut_extension}, as we are examining a toric action defined by $\zeta^{\vec w}$ with $\zeta=-1$, and $w$ being the indicator of the vanishing of $\pt$, we need only show that there exists a triangulation where $\vec w$ is in the mod 2 kernel of $B^T$. 
    \par In the language of arcs, each $\gamma$ belongs to exactly 2 triangles, $\{\gamma_1,\gamma_2,\gamma\}$ and $\{\gamma_3,\gamma_4,\gamma\}$ (i.e. the quiver has arrows $\gamma_2\to\gamma$, $\gamma_4\to \gamma$, $\gamma\to\gamma_1$, $\gamma\to \gamma_3$), and as mod 2 sum and difference agree, it suffices to show $|\{i|\pt(\gamma_i)=0\}|$ is even. We will show this diagrammatically, as, by Lemma \ref{lem: 1-or-3}, we know that every triangle in every triangulation will have one or three sides killed by $\pt$, so the only possible quadrilaterals are given in Figure \ref{fig: quad}. We note that each of these quadrilaterals does in fact have an even number of boundary edges with $\pt(\gamma_i)=0$ (as shown in blue and dashed).
    \par Thus, as long as $\pt$ is deep, we have that $(-1)^{\vec w}$ extends to a unique cluster automorphism $\sgclaut$ on $\A_\surf$ for any fixed triangulation. Note that for any chosen triangulation, we get such an extension.
    \par Now, if two triangulations differ by a flip, we will show that the corresponding automorphisms agree. Note that the automorphism is defined on $\C(x_1,...,x_{n+m})$, so we need only show that the scaling on the new cluster variable is uniquely determined to agree with the definition. Note that we can check the 4 cases in Figure \ref{fig: quad}. Note for the first 3 vanishing patterns, the product of the scaling factors $\zeta\zeta'=1$, whereas for the fourth pattern $\zeta\zeta'=-1$. Knowing this and $\zeta$ uniquely fixes $\zeta'$. Furthermore, in each of the 4 cases, we have that the required $\zeta'$ agrees with $\sgclaut$. Thus, the extension of $\sgclaut$ from the initial seed agrees with the extension of $\sgclaut$ after one mutation.
    \par Now note that any two triangulations can be obtained via a sequence of mutations, so the extension of $\sgclaut$ defines the same cluster automorphism regardless of the initial seed. Thus $\sgclaut$ is a well-defined cluster automorphism on $\surf$, as desired.
\end{proof}
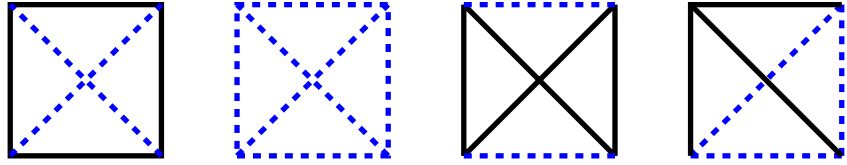
\begin{figure}
    \centering
    \begin{tikzpicture}[line width = 2pt]
    \draw (0,0) -- (2,0) -- (2,2) -- (0,2) -- cycle;
    \draw (0,0)--(2,2) [color=blue][dashed];
    \draw (2,0)--(0,2) [color=blue,dashed];
    
    \draw (3,0) -- (5,0) -- (5,2) -- (3,2) -- cycle [color=blue,dashed];
    \draw (3,0)--(5,2) [color=blue,dashed];
    \draw (5,0)--(3,2) [color=blue,dashed];

    
    \draw (6,0) -- (6,2) (6,2) -- (8,0) (8,0) -- (8,2) (8,2) -- (6,0) [color=black]; 
    \draw (6,0)--(8,0) [color=blue,dashed];
    \draw (6,2)--(8,2) [color=blue,dashed];
    
    \draw (9,0) -- (11,0) (11,0) -- (11,2) (11,2) -- (9,0)[color=blue,dashed];
    \draw (9,0)--(9,2)--(11,2)  [color=black];
    \draw (9,2)--(11,0) [color=black];
\end{tikzpicture}
    \caption{The four possible vanishing patterns of a deep point $\pt$ on a quadrilateral due to the 1-or-3 condition on triangles. Blue dashed arcs represent arcs $\gamma$ with $\pt(\gamma)=0$ while black arcs represent arcs $\gamma$ with $\pt(\gamma)\ne0$.}
    \label{fig: quad}
\end{figure}

Noting that $\sgclaut$ fixes $p$ by definition, we have shown that $p$ can not be mysterious, and as $p$ was an arbitrary deep point of a puncture-free surface type cluster variety, we have completed the proof of Proposition \ref{prop: unpunc}.
\section{The punctured surface case}
\label{sec:ps}
In this section, we will prove the following propositions.
\begin{proposition}
    Given a point $\pt\in\X_\surf$, if there is some puncture $\pun\in\surf$ which is $\pt$-essential, then $\pt$ is not a mysterious point.\label{prop: essential}
\end{proposition}
\begin{proposition}
    Given a point $\pt\in\X_\surf$, if all punctures $\pun\in\surf$ are not $\pt$-essential, then $\pt$ is not a mysterious point. \label{prop: cuttable}
\end{proposition}

After proving the above propositions, Theorem \ref{thm} will follow. To prove Proposition \ref{prop: essential}, we will prove the following lemma about the \emph{Horocycle Rescaling map}.
\begin{definition}
    The \emph{Horocycle Rescaling map} at a puncture $\pun$ with $\zeta\in\C^\times$ is a map $\reclaut:\A_\surf\to\A_\surf$ defined as follows. For each curve $\gamma$, if $\gamma$ is not incident to $\pun$, $\reclaut(\gamma)=\gamma$, if $\gamma$ has one plain end at $\pun$ then $\reclaut(\gamma)=\zeta\gamma$, if it has two plain ends then $\reclaut(\gamma)=\zeta^2\gamma$, if $\gamma$ has one notched end at $\pun$ then $\reclaut(\gamma)=\zeta^{-1}\gamma$, and if $\gamma$ has two notched ends at $\pun$ then $\reclaut(\gamma)=\zeta^{-2}\gamma$.\label{def: horocycle_map}
\end{definition}
\begin{lemma}
    The Horocycle Rescaling map is a cluster automorphism.
\end{lemma}
\begin{proof}
    Note, by Proposition \ref{prop: claut_extension}, it suffices to find a triangulation of $\surf$ such that $\vec w\in\ker(B^T)$, where $w$ is the indicator of incidence to $\pun$ (with $w$ taking on a value of $-1$ on an arc notched at $\pun$), and then show that the obtained map is independent of choice of triangulation.
    \par For an arc $\gamma$ we must test if the $\gamma$ row of $B^Tw$ is 0. If we do that, we show that, for any triangulation, the horocycle rescaling map obeys condition Eqn \ref{eqn:Bwzero}. Now note that there are 3 possibilities for a mutable arc on a surface, up to endpoint identification, as depicted in Figure \ref{fig:three_ways}. 
    \par The row of $B^Tw$ is the sum over the vertices incident to the specified vertex of the weight of $w$. Note that the monomials in the cluster relation are given by the in and out edges respectively. These terms exactly correspond to the positive and negative part of $B^Tw$. So if the monomials in the cluster relations are all scaled by the same total scaling factor, for all relations, then $B^Tw=\vec 0$.
    \par If we restrict this to a Horocycle rescaling map, the above condition is exactly that each monomial in each cluster relation has the same incidences to each puncture and marked point. We need only ensure that each cluster relation is "balanced" in the sense that each monomial in the cluster relation counts the same incidences to the marked points and punctures. Let us examine the cluster relations. 

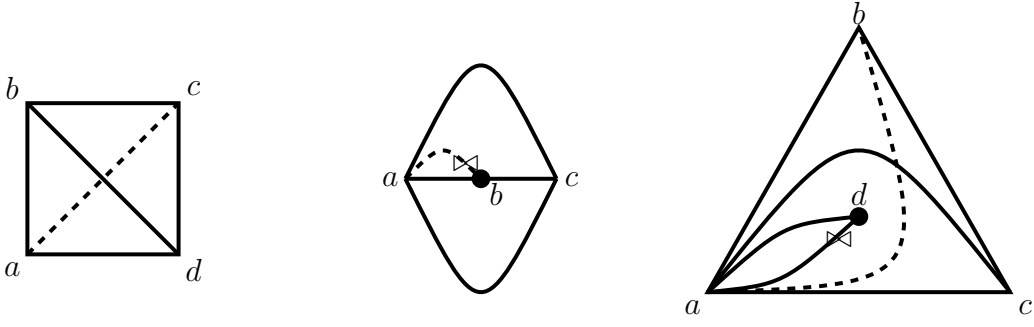
\begin{figure}
        \centering
        \begin{tikzpicture}[line width = 1.5 pt, scale=2]
            \draw (-0.5,-0.5)--(0.5,-0.5)--(0.5,0.5)--(-0.5,0.5)--cycle;
            \draw (0.5,-0.5)--(-0.5,0.5);
            \draw[dashed] (-0.5,-0.5)--(0.5,0.5);
            \node[draw, circle, fill=black, scale=0.5] (A0) at (2.5,0) { };
            \draw (2,0)--(2.5,0)--(3,0);
            \draw (2,0)..controls (2.5,1)..(3,0);
            \draw (2,0)..controls (2.5,-1)..(3,0);
            \draw[dashed] (2,0)..controls (2.25,0.25)..(2.5,0);
            \node (l1) at (-0.6,-0.6) {$a$};
            \node (l2) at (-0.6,0.6) {$b$};
            \node (l3) at (0.6,0.6) {$c$};
            \node (l4) at (0.6,-0.6) {$d$};
            
            \node (A1) at (2.4,0.1) { $\bowtie$};
            \node (l5) at (1.9,0) {$a$};
            \node (l6) at (2.6,-0.1) {$b$};
            \node (l7) at (3.1, 0) {$c$};

            \draw ( 4,-0.75) -- (6, -0.75) -- (5,1)--cycle;
            \node[draw, circle, fill=black, scale=0.5] (A2) at (5,-.25) { };
            \draw( 4,-0.75)..controls (5,0.5).. (6, -0.75);
            \draw[dashed]( 4,-0.75)..controls (5.5, -0.7)..(5,1);
            \draw( 4,-0.75)..controls(4.5,-0.3 )..(5,-.25);
            \draw( 4,-0.75)..controls(4.5,-0.7 )..(5,-.25);
            \node (A3) at (4.87,-0.4) { $\bowtie$};
            
            \node (l8) at (3.9, -0.85) { $a$};
            \node (l9) at (5, 1.1) { $b$};
            \node (l10) at (6.1, -0.85) { $c$};
            \node (l11) at (5,-.1) {$d$};

        \end{tikzpicture}
        \caption{The three types of cluster relations for a tagged surface. All other cluster relations are foldings of these. The dashed line represents the mutated cluster variable. Letters indicate the labels used in the proof.}
        \label{fig:three_ways}
    \end{figure}
    For the standard relation, we have the relation:
    $$\gamma_{bd}\gamma_{ac}=\gamma_{ab}\gamma_{cd}+\gamma_{bc}\gamma_{ad}$$
    Note that all monomials have the same total letters, and thus this relation is satisfied by the horocycle rescaling map. Next, the folded relation.
    $$\gamma_{bc}\cdot\gamma_{ab^{\bowtie}}=\gamma_{ac}+\gamma_{ac} '$$
    where $\gamma_{ac}$ goes below the puncture, and $\gamma_{ac}'$ goes above the puncture. This is balanced in the letters, as $b$ and $b^{\bowtie}$ cancel (i.e. when scaling by $\reclaut$, we will have a term $\zeta$ and $\zeta^{-1}$). Finally, the diagonal flip around a notched puncture.
    $$\gamma_{ab}'\gamma_{ac}'=\gamma_{ab}\gamma_{ac}+\gamma_{ad}\gamma_{ad^{\bowtie}} \gamma_{bc},$$
    where $\gamma_{ab}'$ and $\gamma_{ac}'$ are the arcs which trace around the puncture. We note that, as $d$ and $d^{\bowtie}$ cancel, all terms are labeled $a^2bc$.

    \par Thus, all cluster relations have balanced incidence to each marked point for the three types of cluster relations present on marked surfaces. If we were to identify any marked points in any of these relations, the balanced condition still holds. Thus we have shown that for each arc $\gamma$ on the surface, $B^Tw|_\gamma=0$ for $w$ defined by Definition \ref{def: horocycle_map}. Thus $w\in \ker(B^T)$, meaning the horocycle rescaling map is a cluster automorphism, as desired.
\end{proof}
\par Note that $\reclaut$ fixes $\pt$, as $\pt(\gamma)=0$ on all arcs $\gamma$ incident to $\pun$ by the definition of a $\pt$-essential puncture. Thus $\pt$ can not be a mysterious point, finishing the proof of Proposition \ref{prop: essential}.
\par Finally, we need to show Proposition \ref{prop: cuttable}. To do this, we will need a few lemmas. The main tool in this section will be cutting along arcs $\gamma$ with $\pt(\gamma)\ne0$.
\begin{lemma}
    Given a surface $\surf$ and an arc $\gamma$, then if $\surf'$ is the surface given by cutting along $\gamma$, then the variety $\X_{\surf'}|_{\gamma_1=\gamma_2}$ where $\gamma_1$ and $\gamma_2$ are the images of the cut, includes into $\X_\surf$ as a cluster chart as long as $\X_\surf$ is locally acyclic. \label{lem:local_chart}
\end{lemma}
\begin{proof}
    The variety $\X_{\surf'}|_{\gamma_1=\gamma_2}$ is isomorphic to the variety over the cluster algebra given by freezing at $\gamma$. This result then becomes a restatement of \cite{Muller2013LocallyAcyclic}, Lemma 3.4. This is also a result of \cite[Proposition 6.12]{BM}.
\end{proof}

Next, we will work on constructing convenient curves $\gamma$ to cut along.
\begin{lemma}
    (\cite{ASS12}, Lemma 4.9) For $\pun\in\puns$, the map $\taginv:\A_\surf\to\A_\surf$ which swaps the taggings at $\pun$ is an algebra automorphism.
\end{lemma}
\begin{proof}
    Reversing the taggings at $\pun$ in any seed does not change the quiver. Since this map is an isomorphism on quivers, it lifts to a cluster isomorphism on the level of cluster algebras, and thus is an algebra automorphism.
\end{proof}
\begin{corollary}
    Given $\surf$ and a deep point $\pt$, if $\pt'= \pt\circ \taginv$ admits a cluster automorphism $\claut'\in\Aut(\A_\surf)$ which fixes $\pt'$, then $\claut = \taginv\circ \claut'\circ(\taginv)^{-1}$ is a cluster automorphism which fixes $\pt$.\label{cor: tag_rem}
\end{corollary}
\begin{proof}
    $\claut$ fixes $\pt$ via the following calculation. Note $\pt = \pt'\circ (\taginv)^{-1}$
    $$\claut(\pt)=\pt\circ\vp$$
    $$\claut(\pt)=\pt'\circ (\taginv)^{-1}\circ\taginv\circ \claut'\circ(\taginv)^{-1}$$
    $$\claut(\pt)=\pt'\circ  \claut'\circ(\taginv)^{-1}$$
    Now, we use that $\claut'$ fixes $\pt'$
    $$\claut(\pt)=\pt'\circ(\taginv)^{-1}=\pt$$
    Thus $\claut$ fixes $\pt$. It now suffices to show that $\claut$ rescales cluster variables. Let $\gamma$ be a cluster variable.
    $$\claut(\gamma)=\taginv\circ \claut'\circ(\taginv)^{-1}(\gamma)$$
    Note that if $\gamma$ is a cluster variable, then $(\taginv)^{-1}(\gamma)$ is a cluster variable, $\gamma'$ where $\gamma'$ has the tag at $\pun$ inverted (if the boundary of $\gamma$ includes $\pun$, otherwise, $\gamma'=\gamma$).  
    $$\claut(\gamma)=\taginv\circ \claut'(\gamma')=\taginv(\zeta\gamma')$$
    Note that $\taginv$ flips the tag of $\gamma'$ back.
    $$\claut(\gamma)=\zeta\gamma$$
    Thus $\claut$ rescales the cluster variables. Thus $\claut$ is a cluster automorphism which fixes $\pt$, as desired.
    
\end{proof}
The corollary allows us to convert tags from notched to plain at a puncture.
\begin{lemma}
    \label{lem: bndy_cut} If $\pun$ is not $\pt$-essential, then there exists some curve $\gamma$ connecting $\pun$ to the boundary with $\pt(\gamma)\ne 0$. Moreover, this can be chosen to avoid any closed subset $D$ so long as $\surf\smallsetminus D$ is connected and contains at least 1 boundary point. \label{lem: non_zero_bndy_arc}
\end{lemma}
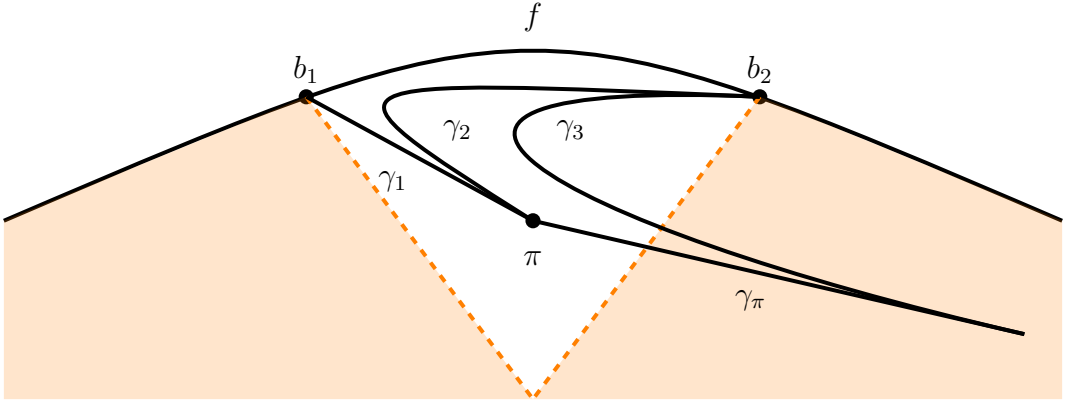
\begin{figure}
    \centering
    \begin{tikzpicture}[line width = 1.5pt]
        \draw (-7,0) .. controls (0, 3) .. (7,0);
        \node (f) at (0,2.7) {$f$};
        \fill (-3, 1.64) circle [radius = 0.1];
        \fill (3, 1.64) circle [radius = 0.1];
        \node (a) at (-3, 2) {$\bn_1$};
        \node (b) at (3, 2) {$\bn_2$};
        \draw[dashed][color=orange] (-3, 1.64) -- (0,1.64-4);
        \draw[dashed][color=orange] (3, 1.64) -- (0,1.64-4);
        \fill[orange, fill opacity = 0.2] (-3, 1.64) -- (-0,1.64-4) -- (-7,1.64-4) -- (-7,0)--cycle;
        \fill[orange, fill opacity = 0.2] (3, 1.64) -- (0,1.64-4)-- (7,1.64-4) -- (7,0)--cycle;

        \fill (0,0) circle [radius = 0.1];
        \node (q) at (0,-0.5) {$\pun$};
        \draw (0,0) to [edge label=$\gamma_1$] (-3, 1.64);
        \draw (0,0) to [edge label' = $\gamma_\pun$] (6.5,-1.5);
        \draw (0,0) .. controls (-3, 1.9) .. (3, 1.64);
        \node (c) at (-1, 1.2) {$\gamma_2$};
        \draw (6.5,-1.5) .. controls (0,0) and (-3, 1.9).. (3, 1.64);
        \node (d) at (0.5, 1.2) {$\gamma_3$};

    \end{tikzpicture}
    \caption{The Orange region represents $D$. Note that the marked boundary points can be on the boundary of $D$.}
    \label{fig: cut_to_bndy}
\end{figure}
\begin{proof}
    Via Figure \ref{fig: cut_to_bndy}, we can construct an arc $\gamma_1$ from $\pun$ to the boundary $b_1$ not in $D$, we then take this arc and trace along it followed by the boundary to the next point $b_2$(possibly the same point), not in $D$, creating the arc $\gamma_2$. Note $\gamma_1\ne\gamma_2$ and they are compatible arcs. Finally, we take the arc $\gamma_\pun$ with $\pt(\gamma_\pun)\ne0$ given by assuming $\pun$ is not $\pt$-essential, and the arc $\gamma_3$ will be the arc obtained by following $\gamma_\pun$ then $\gamma_2$. Now we note that if $f$ is the arc connecting $b_1$ with $b_2$, then $(\gamma_1,f,\gamma_3,\gamma_\pun)$ is a quadrilateral, with diagonal $\gamma_2$. Mutating at $\gamma_2$ gives $$\gamma_2\gamma_2'=\gamma_1\cdot\gamma_3+f\cdot\gamma_\pun$$
    Noting that $\pt(f\cdot\gamma_\pun)\ne0$, we must have that either $\pt(\gamma_1)\ne0$ or $\pt(\gamma_2)\ne0$. As both $\gamma_1$ and $\gamma_2$ avoid $D$ and connect $\pun$ to the boundary, this completes the proof.
\end{proof}
\begin{lemma}
    Given $\pt\in\X_\surf$, if every puncture (minimum 1) $\pun\in\surf$ is not $\pt$-essential, there exists a collection of mutually compatible arcs $\gamma_\pun$ indexed by punctures with $\pt(\gamma_\pun)\ne 0$, such that, all punctures are removed by cutting along each arc in the collection, and such a cutting does not disconnect the surface. \label{lem: cuts}
\end{lemma}
\begin{proof}

First, if the surface does not have a boundary with at least 2 marked points, note that there exists a boundary with at least one marked point due to the locally acyclic condition \cite[Theorem 10.9, Theorem 10.10]{Muller2013LocallyAcyclic}. Next take any puncture $\pun$, note it is not $\pt$-essential, so it admits an arc $\gamma_\pun$ with $\pt(\gamma_\pun)\ne0$. Now, we can use Lemma \ref{lem: non_zero_bndy_arc} to find $\gamma_\pun$ with $\pt(\gamma_\pun)\ne0$ and $\gamma_\pun$ connects $\pun$ to the boundary. We can take $\gamma_\pun$ to be plain at $\pun$ via Corollary \ref{cor: tag_rem}. Cutting along $\gamma_\pun$ gives us a boundary with 3 marked points. Note that if the new surface $\surf'$ admits the desired collection $\Gamma$, then the original surface $\surf$ admits $\Gamma\cup\gamma_\pun$, as $\gamma_\pun$ is a boundary in $\surf'$, all arcs in $\Gamma$ will have images in $\surf$ compatible with $\gamma_\pun$.

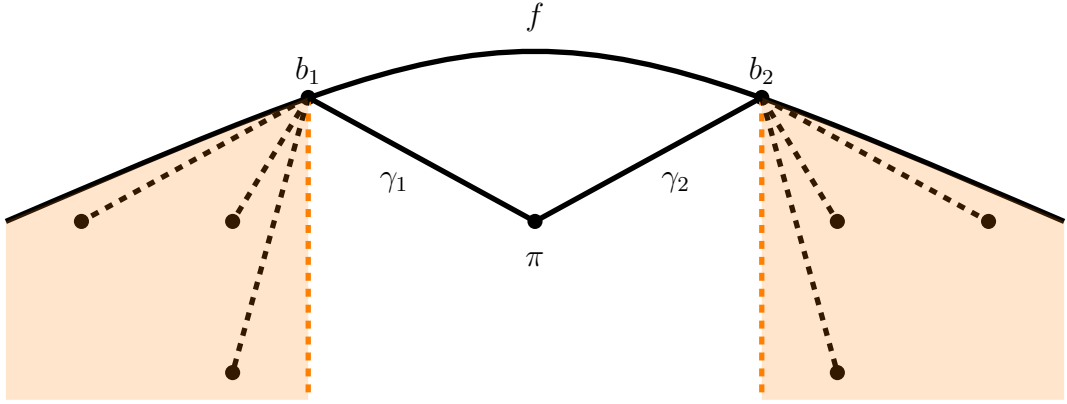
\begin{figure}[!h]
    \centering
    \begin{tikzpicture}[line width = 2pt]
        \draw (-7,0) .. controls (0, 3) .. (7,0);
        \node (f) at (0,2.7) {$f$};
        \fill (-3, 1.64) circle [radius = 0.1];
        \fill (3, 1.64) circle [radius = 0.1];
        \node (a) at (-3, 2) {$\bn_1$};
        \node (b) at (3, 2) {$\bn_2$};
        \draw[dashed][color=orange] (-3, 1.64) -- (-3,1.64-4);
        \fill (-4,0) circle [radius = 0.1];
        \fill (-4,-2) circle [radius = 0.1];
        \fill (-6,0) circle [radius = 0.1];
        \draw[dashed] (-3, 1.64) -- (-4,0);
        \draw[dashed] (-3, 1.64) -- (-4,-2);
        \draw[dashed] (-3, 1.64) -- (-6,0);
       
        \draw[dashed][color=orange] (3, 1.64) -- (3,1.64-4);
        \fill (4,0) circle [radius = 0.1];
        \fill (4,-2) circle [radius = 0.1];
        \fill (6,0) circle [radius = 0.1];
        \draw[dashed] (3, 1.64) -- (4,0);
        \draw[dashed] (3, 1.64) -- (4,-2);
        \draw[dashed] (3, 1.64) -- (6,0);

        \fill[orange, fill opacity = 0.2] (-3, 1.64) -- (-3,1.64-4) -- (-7,1.64-4) -- (-7,0)--cycle;
        \fill[orange, fill opacity = 0.2] (3, 1.64) -- (3,1.64-4)-- (7,1.64-4) -- (7,0)--cycle;

        \fill (0,0) circle [radius = 0.1];
        \node (q) at (0,-0.5) {$\pun$};
        \draw (0,0) to [edge label=$\gamma_1$] (-3, 1.64);
        \draw (0,0) to [edge label'=$\gamma_2$] (3, 1.64);

    \end{tikzpicture}
    \caption{The inductive step. Dashed black arcs represent the arcs previously chosen. The orange dashed lines indicate the boundary of $D$. By Lemma \ref{lem: bndy_cut}, we have either $\pt(\gamma_1)\ne0$ or $\pt(\gamma_2)\ne0$.}
    \label{fig: inductive_step}
\end{figure}
\par Now, assuming that there is a boundary with at least 2 marked points, we will use Lemma \ref{lem: bndy_cut} repeatedly, inducting on the number of punctures, adding the newly created arcs to the set $D$. Visually, this step is done as in Figure \ref{fig: inductive_step}. Noting by induction that all arcs in $D$ are mutually compatible, the new arc is compatible with all previous arcs by noting that it does not intersect $D$. Thus by induction, we have a set of mutually compatible arcs $\gamma_\pun$ for each $\pun \in\puns$ with $\pt(\gamma_\pun)\ne 0$, as desired. Furthermore, as no arc alone disconnects the surface, the collection will not by induction, and as each cut cuts a puncture to the boundary, after cutting along all arcs, no punctures remain, thus completing the proof.
\end{proof}

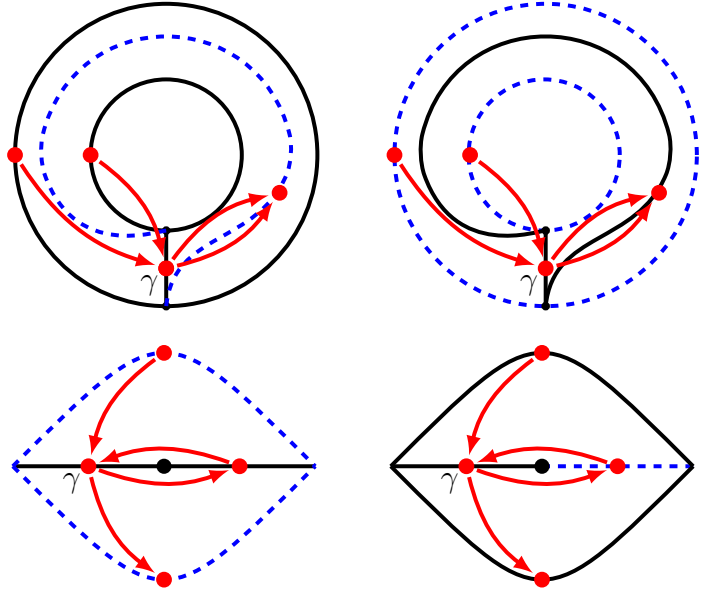
\begin{figure}
    \centering
    
\begin{tikzpicture}[line width = 1.5pt,scale=2]
    \def\width{3}
    \def\Rout{1}   
    \def\Rin{0.5}  
    \node[draw, circle, scale = 0.4, red, fill= red, outer sep = 5pt](A1) at (0,-0.5*\Rin-0.5*\Rout) { };
    \draw[black] (0,0) circle (\Rout);
    \draw[black] (0,0) circle (\Rin);

    \filldraw[black] (0,-\Rout) circle (0.015);
    \filldraw[black] (0,-\Rin) circle (0.015);

    \draw[black] (0,-\Rout) -- (0,-\Rin);

    \draw[blue, dashed, line width=1.5pt, rounded corners=8pt]
        (0,-\Rin) 
        .. controls (-0.05, -0.5) and (-0.6, -0.7) ..
        (-0.85,0)
        .. controls (-0.6,1.0) and (0.6,1.0) ..
        (0.85,0)
        .. controls (0.6,-0.6) and (0.05,-0.5) ..
        (0,-\Rout);
    \node[draw, circle, scale = 0.4, red, fill= red, outer sep = 5pt](A1) at (0,-0.5*\Rin-0.5*\Rout) { };
    \node[draw, circle, scale = 0.4, red, fill= red, outer sep = 5pt](A2) at (0.5*\Rin+0.5*\Rout,-.25) { };
    \node[draw, circle, scale = 0.4, red, fill= red, outer sep = 5pt](A3) at (-\Rin,0) { };
    \node[draw, circle, scale = 0.4, red, fill= red, outer sep = 5pt](A4) at (-\Rout,0) { };
    
    \draw[-{latex},red] (A3) to[bend left=20] (A1);
    \draw[-{latex},red] (A4) to[bend right=20] (A1);
    \draw[-{latex},red] (A1) to[bend left=20] (A2);
    \draw[-{latex},red] (A1) to[bend right=20] (A2);

    \draw[blue,dashed] (2.5\Rout,0) circle (\Rout);
    \draw[blue,dashed] (2.5*\Rout,0) circle (\Rin);

    \filldraw[black] (02.5\Rout,-\Rout) circle (0.015);
    \filldraw[black] (02.5\Rout,-\Rin) circle (0.015);

    \draw[black] (02.5\Rout,-\Rout) -- (02.5\Rout,-\Rin);

    \draw[black, line width=1.5pt, rounded corners=8pt]
        (2.5\Rout,-\Rin) 
        .. controls (2.5\Rout-0.05, -0.5) and (2.5\Rout-0.6, -0.7) ..
        (2.5\Rout-0.85,0)
        .. controls (2.5\Rout-0.6,1.0) and (2.5\Rout+0.6,1.0) ..
        (2.5\Rout+0.85,0)
        .. controls (2.5\Rout+0.6,-0.6) and (2.5\Rout+0.05,-0.5) ..
        (2.5\Rout+0,-\Rout);
\node[draw, circle, scale = 0.4, red, fill= red, outer sep = 5pt](B1) at (2.5\Rout+0,-0.5*\Rin-0.5*\Rout) { };
    \node[draw, circle, scale = 0.4, red, fill= red, outer sep = 5pt](B2) at (2.5\Rout+0.5*\Rin+0.5*\Rout,-.25) { };
    \node[draw, circle, scale = 0.4, red, fill= red, outer sep = 5pt](B3) at (2.5\Rout-\Rin,0) { };
    \node[draw, circle, scale = 0.4, red, fill= red, outer sep = 5pt](B4) at (2.5\Rout-\Rout,0) { };
    
    \draw[-{latex},red] (B3) to[bend left=20] (B1);
    \draw[-{latex},red] (B4) to[bend right=20] (B1);
    \draw[-{latex},red] (B1) to[bend left=20] (B2);
    \draw[-{latex},red] (B1) to[bend right=20] (B2);
    \node (gam1) at (-0.11,-0.5*\Rin-0.5*\Rout-0.11) { $\gamma$ };
    \node (gam2) at (2.5\Rout-0.11,-0.5*\Rin-0.5*\Rout-0.11) { $\gamma$ };
\end{tikzpicture}\\
\begin{tikzpicture}[line width = 1.5pt,scale=2]

        \draw (4,0)--(5,0);
        \draw (5,0)--(6,0)[color=black];
        \fill (5,0) circle [radius = 0.05];
        \draw (4,0) .. controls (5,1) .. (6,0) [color=blue,dashed];
        \draw (4,0) .. controls (5,-1) .. (6,0) [color=blue,dashed];
        \node[draw, circle, scale = 0.4, red, fill= red, outer sep = 5pt](A1) at (4.5,0) { };
        
        \node[draw, circle, scale = 0.4, red, fill= red, outer sep = 5pt](A2) at (5.5,0) { };
        \node[draw, circle, scale = 0.4, red, fill= red, outer sep = 5pt](A3) at (5,0.75) { };
        \node[draw, circle, scale = 0.4, red, fill= red, outer sep = 5pt](A4) at (5,-0.75) { };

        \draw[-{latex},red] (A1) to[bend right=20] (A4);
        \draw[-{latex},red] (A3) to[bend right=20] (A1);

        \draw[-{latex},red] (A1) to[bend right=20] (A2);
        \draw[-{latex},red] (A2) to[bend right=20] (A1);

        \draw (6.5,0)--(7.5,0);
        \draw (7.5,0)--(8.5,0)[color=blue, dashed];
        \fill (7.5,0) circle [radius = 0.05];
        \draw (6.5,0) .. controls (7.5,1) .. (8.5,0) [color=black];
        \draw (6.5,0) .. controls (7.5,-1) .. (8.5,0) [color=black];
        \node[draw, circle, scale = 0.4, red, fill= red, outer sep = 5pt](B1) at (7,0) { };
        \node[draw, circle, scale = 0.4, red, fill= red, outer sep = 5pt](B2) at (8,0) { };
        \node[draw, circle, scale = 0.4, red, fill= red, outer sep = 5pt](B3) at (7.5,0.75) { };
        \node[draw, circle, scale = 0.4, red, fill= red, outer sep = 5pt](B4) at (7.5,-0.75) { };
        \draw[-{latex},red] (B1) to[bend right=20] (B4);
        \draw[-{latex},red] (B3) to[bend right=20] (B1);

        \draw[-{latex},red] (B1) to[bend right=20] (B2);
        \draw[-{latex},red] (B2) to[bend right=20] (B1);
        
        \node (gam1) at (4.5-0.11,-0.11) { $\gamma$ };
    \node (gam2) at (7-0.11,-0.11) { $\gamma$ };
\end{tikzpicture}

    \caption{Above are the only foldings of \ref{fig: quad} with a diagonal $\gamma$ with $\pt(\gamma)\ne0$. Note that edges $\gamma_1,\gamma_2$ which are folded must either have $\pt(\gamma_1)\ne 0$ and $\pt(\gamma_2)\ne0$ or $\pt(\gamma_1)=\pt(\gamma_2)=0$. Note that the double foldings result in the once-punctured torus and the thrice-punctured sphere. The quiver of the latter is empty, the quiver of the former is the Markov quiver, which is empty up to mod 2.}
    \label{fig: folded_rects}
\end{figure}

\begin{lemma}
    Let $\surf$ be a surface, $\pt\in\X_\surf$ a deep point, and $\gamma_i,$ $i\in I$ with $\pt(\gamma_i)\ne0$ a collection of compatible arcs. Let $\surf'$ be the surface obtained by cutting along $\gamma_i$, and assume that $\surf'$ is connected with no punctures, with $\pt'$ being the corresponding point when restricting $\pt$ to $\X_{\surf'}$. Then, the deep sign-reversing cluster automorphism $(\sgclaut)'$ extends uniquely to a cluster automorphism on $\A_\surf$. 
    \label{lem:ext}
\end{lemma}
\begin{proof}
    We proceed by induction. We start by extending $\gamma_i$ to a triangulation of $\surf$ with only plain arcs. We note that the only quadrilaterals that a curve $\gamma_i$ can be a part of are given in Figure \ref{fig: quad}, or such quadrilaterals with some sides identified. The only ways to identify sides are given in Figure \ref{fig: folded_rects}. Thus, we note that the $B$ matrix row associated to $\gamma_i$ will have values given by one of these figures. Order the $B$ matrix so that the $|I|$ curves are the last mutable rows/columns. Furthermore, we note that up to mod 2, the amount of sides (counted with appropriate multiplicity) which are killed by a deep point $\pt$ is 0, thus if $\vec w|_{[n-|I|+i-1]}\in \ker(B^T|_{[n-|I|+i-1]})$ mod $2$, then $\vec w|_{[n-|I|+i]}\in \ker(B^T|_{[n-|I|+i]})$ mod 2, and as we know that when there are no punctures, i.e. the base case, that $w|_{[n-|I|]}\in\ker(B^T)|_{n-|I|}$, we have that $w\in\ker(B^T)$, and thus $(\sgclaut)'$ extends uniquely to $\sgclaut$ on $\surf$.
\end{proof}  
Here it is important to note that the action of $\sgclaut$ on any curve not in the original triangulation can not be dictated, i.e. on notched arcs, for instance, it may not reverse the sign of a notched arc $\gamma$ with $\pt(\gamma)=0$.

\begin{lemma}
    If $\A'$ is a freezing of $\A$ at $\{a_1,...,a_i\}$ and both are locally acyclic, then any cluster automorphism $\claut$ of $\A$ takes points in $\X_{\A'}$ to points in $\X_{\A'}$. \label{lem:fixed_sub}
\end{lemma}
\begin{proof}
    As in \cite{Muller2013LocallyAcyclic}, Lemma 3.4, $\X_{\A'}$ is the open subscheme of $\X_\A$ where $a_1,...,a_i$ are non-zero. If $\claut$ is a cluster automorphism, then for each cluster variable $a$, $\claut(a)=\zeta a$ for some $\zeta\in\C^\times$. Thus if $\pt$ is a point in $\X_{\A}$, for any cluster variable $a$, $\pt(a)=0$ if and only if $\vp(\pt)(a)=0$. Thus $\pt$ is in the subscheme where $a_1,...,a_i$ are non-zero if and only if $\claut(\pt)$ is.  
\end{proof}

\begin{lemma}
    If $\X_\surf$ is locally acyclic, and $\pt\in \X_\surf$ is a deep point with all punctures $\pun\in\puns\subset\surf$ not $\pt$-essential then the cluster automorphism $\sgclaut$ fixes $\pt$.
\end{lemma}
\begin{proof}
    Given $\pt$ as above, using Lemmas \ref{lem:ext}, \ref{lem: cuts}, we can find a cluster automorphism $\sgclaut$ such that the induced point $\pt'$ on $\surf'$, the surface obtained by cutting along all cuts, is fixed by $(\sgclaut)'$. Now, using Lemma \ref{lem:local_chart}, we have that the subscheme of $\X_{\surf'}$ given by restricting to points which agree on both images of each cut includes as a cluster chart into $\X_\surf$. Furthermore, $\sgclaut$ takes points in this chart to other points in this chart by Lemma \ref{lem:fixed_sub}. Furthermore, we have that, by the definition of $\sgclaut$, it agrees with $(\sgclaut)'$ on this chart. Moreover, we know,  as $\pt$ is non-zero along the cuts, it belongs to this cluster chart, and it in fact agrees with $\pt'$. Thus we have that $\sgclaut$ acts as $(\sgclaut)'$ on this chart, and $\pt$ is in this chart and agrees with $\pt'$, we have that $\pt$ must be fixed by $\sgclaut$, as $\pt'$ is fixed by $(\sgclaut)'$. 
\end{proof}
Thus, if all punctures are not $\pt$-essential, we have constructed a cluster automorphism $\sgclaut$ which fixes $\pt$, and thus $\pt$ can not be a mysterious point, thus finishing the proof of Proposition \ref{prop: cuttable}.
\par Now that we have shown Propositions \ref{prop: unpunc}, \ref{prop: essential}, and \ref{prop: cuttable}, we see that no locally acyclic surface type cluster algebras admit mysterious points, thus showing Theorem \ref{thm}.
\begin{corollary}
    If $\surf=\bigsqcup_i\surf_i$ i.e. $\surf$ is a marked surface with multiple connected components, as long as $\X_{\surf_i}$ is locally acyclic for all $i$, then $\X_\surf$ has no mysterious points.
    \label{cor:multi}
\end{corollary}
\begin{proof}
    Let $\pt\in\X_\surf=\X_{\bigsqcup_i\surf_i}=\prod_i\X_{\surf_i}$ be deep. Assume that for all $i$, $\pt|_{\surf_i}$ is not deep. Then there exists a triangulation $T_i$ of $\surf_i$ for each $i$ such that $\pt(\gamma)\ne0$ for all $\gamma$ in that triangulation. Take the triangulation $T$ of $\surf$ as $\bigcup_i T_i$ the union of these triangulations. Then note $\pt(\gamma)\ne0$ for all $\gamma\in T$. Thus $\pt$ is not deep, which is a contradiction. Thus $\pt|_{\surf_i}$ must be deep for some $i$.
    \par Given $\pt|_{\surf_i}$ is deep, by Theorem \ref{thm}, there is some $\claut'$ on $\X_{\surf_i}$ which fixes $\pt|_{\surf_i}$, then we can define $\claut$ to be the identity on the other components of the surface and $\claut'$ on $\X_{\surf_i}$, i.e. $\claut=\claut'_i\times\prod_{j\ne i}\mathds1_j$. Now we have:
    $$\claut(\pt)=\claut'_i(\pt|_{\surf_i})\times\prod_{j\ne i}\mathds1_j(\pt|_{\surf_j})=\pt|_{\surf_i}\times\prod_{j\ne i}\pt|_{\surf_j}=\pt.$$
    Thus given $\pt$ is deep, we can find $\claut$ which fixes $\pt$, and thus $\X_\surf$ has no mysterious points.
\end{proof}

\section{Part 2: Finite Type Cluster Algebras}
\label{sec:p2}
The no mysterious point conjecture has been shown to hold for simply-laced finite type cluster algebras (i.e. types A, D, and E) by \cite[Corollary 5.20]{CGSS}. The goal of this section is to prove the result for types B and C, and then the two exceptional cases $F_4$ and $G_2$. These cases combined with the work of Castronovo, Gorsky, Simental, and Speyer will complete the proof of the following theorem.
\begin{theorem}
    Finite-type Cluster Algebras have no mysterious points.
\end{theorem}
The following lemma will prove instrumental.
\begin{lemma}
    Let $Q$ be a quiver, $v\in Q$ a vertex, and let $v'+Q=Q'$ be a new quiver identical to $Q$ but with a leaf $v'$ added incident to $v$. Then, if $p\in \X_{Q'}$ is a point such that $p(v)\in\C^\times$, and $\pt|_{Q-v}$ is not deep, then $p$ is not deep. \label{lem:not_deep_leaf}
\end{lemma}
\begin{proof}
    We note that we can find a chart on $Q'$ such that $p|_{Q-v}$ is non-zero by noting that freezing at $v$ would give us the variety $\X_{Q-v}\times\X_{A_1}$, and $\pt$ restricted to $\X_{Q-v}$ is not deep, so we can find a cluster on which the point evaluates to non-zero on each cluster variable. We can use the same mutations to get to this chart on $Q'$, without mutating at $v$. And since $v'$ is a leaf at $v$, the result of the mutations will still be $v'$ being a leaf at $v$. This means that we can find a cluster chart of $Q'$ with all vertices but $v,v'$ frozen, and $p$ will be in this chart. But, we know that the cluster variety of this chart is isomorphic to $A_2$ cross some torus, as there are only 2 mutable vertices. However, $A_2$ has no deep points, meaning for every point in this chart, there exists a mutation equivalent quiver such that $p$ on every cluster variable in that quiver is non-zero. Thus $p$ can not be a deep point in this chart, meaning it lies in a cluster torus in this chart. Note that the cluster tori of the charts are a subset of the cluster tori of the whole variety. Thus $p$ is in a cluster torus, and is thus not a deep point, as desired. 
\end{proof}
First we will investigate the rank 2 case, followed by types $B$ and $C$ and finally type $F_4$.

\subsection{Rank 2}\label{sec:rank 2}
Other than type $A_2$, which has no deep points due to the combinatorial restriction from Lemma \ref{lem: 1-or-3}, the other two finite type rank 2 cluster algebras are type $B_2$ and type $G_2$. Note these are acyclic cluster algebras, and thus by \cite{BFZ}, if the initial cluster is $(x_1,x_2)$, the cluster algebra $\A\cong\Bk[x_0,x_1,x_2,x_3]/\sim$ where $x_0$ is given by mutation at $x_2$, $x_3$ is given by mutation at $x_1$, and $\sim$ is the ideal generated by these two cluster relations.

\par We will make use of \cite[Theorem 4.5]{BM}.
\begin{proposition}
    \cite[Theorem 4.5]{BM} Let $\A$ be a rank 2 cluster algebra with exchange matrix:
    $$\begin{bmatrix}
        0&b\\-c&0\\-\vec e_2&\vec e_1
    \end{bmatrix}$$
    Assume either $c>1$ or $\operatorname{char}(\Bk)$ divides $b$. For each $\alpha\in\Bk^\times$ and each $\beta\in(\Bk^\times)^f$ such that $\alpha^b+\beta^{-\vec e_1}=0$, there is a deep point in $V(\A,\Bk)$ defined by
    $$(x_0,x_1,x_2,x_3,y_1,...,y_f)\mapsto(0,\alpha,0,\alpha^{-1}\beta^{[\vec e_2]_{-}},\beta_1,...,\beta_f)$$
    Assume either $b>1$ or $\operatorname{char}(\Bk)$ divides $c$. For each $\alpha \in \Bk^\times$ and each $\beta\in(\Bk^\times)^f$ such that $\alpha^c+\beta^{-\vec e_2}=0$, there is a deep point in $V(\A,\Bk)$ defined by
    $$(x_0,x_1,x_2,x_3,y_1,...,y_f)\mapsto (\alpha^{-1}\beta^{[\vec e_1]_{-}},0,\alpha,0,\beta_1,...,\beta_f)$$
    Furthermore, every deep point in $V(\A,\Bk)$ is one of these two mutually exclusive types.
    \label{prop:rk2}
\end{proposition}
Note that the above proposition also handles the $A_2$ case, namely, since neither $b>1$ nor $c>1$, $A_2$ with any number of frozens has no deep points. We will show the following Lemma.
\begin{lemma}
    All deep points arising from Proposition \ref{prop:rk2} are stabilized by cluster automorphisms.
\end{lemma}
As the Proposition ensures that there are no other deep points, this lemma will prove that there are no mysterious points for rank 2 cluster algebras over $\C$.

\par Since all of our calculations will be over $\Bk=\C$, and $\C$ is characteristic 0, we have that the condition for points of the first type to exist is that $c>1$, and the condition for points of the second type to exist is $b>1$. To ensure really full rank, we will take that $\vec e_2=[0,1]$ and $\vec e_1 = [1,0]$. By Proposition \ref{prop:RFR}, this will suffice to show the general case. Thus our exchange matrix becomes:
$$\begin{bmatrix}
    0&b\\-c&0\\0&1\\-1&0
\end{bmatrix}$$
\par We first consider points of the first type. Assume $c>1$. Let $\zeta$ be a $c$th root of unity. Note as $c>1$ we can find such a $\zeta$ with $\zeta\ne1$. Now take $\vec w=[0,1,0,0]$, and notice:
$$B^Tw=\begin{bmatrix}
    0&-c&0&-1\\b&0&1&0
\end{bmatrix}\begin{bmatrix}
    0\\1\\0\\0
\end{bmatrix}=\begin{bmatrix}
    -c\\0
\end{bmatrix}\equiv\vec0 \text{  (mod $c$)}$$
Note that thus the cluster automorphism which multiplies $(x_1,x_2,y_1,y_2)$ by $\zeta^{\vec w}$, i.e. 
$$\vp:\A\to \A$$ $$x_1\mapsto x_1,$$ $$x_2\mapsto \zeta x_2,$$ $$y_1\mapsto y_1,$$ $$y_2\mapsto y_2,$$ extends uniquely to a cluster automorphism of $\A$ by Proposition \ref{prop: claut_extension}. Furthermore, we know from the cluster relations at $x_1$ and $x_2$ that this automorphism must take $x_0\mapsto \zeta^{-1}x_0$ and $x_3\mapsto x_3$. Now let us examine the effect of this automorphism on points of the form
$$\pt:(x_0,x_1,x_2,x_3,y_1,y_2)\mapsto (0,\alpha,0,\alpha^{-1},\beta_1,\beta_2)$$
Note that $\pt\circ\vp=\pt$, and thus if $c>1$ we can find a cluster automorphism which stabilizes points of the first type, as desired.
\par Now assume $b>1$, and let $\zeta$ be a $b$th root of unity. Note as $b>1$ we can find such a $\zeta$ with $\zeta\ne1$. Now take $\vec w=[1,0,0,0]$, and notice:
$$B^Tw=\begin{bmatrix}
    0&-c&0&-1\\b&0&1&0
\end{bmatrix}\begin{bmatrix}
    1\\0\\0\\0
\end{bmatrix}=\begin{bmatrix}
    0\\b
\end{bmatrix}\equiv\vec0 \text{  (mod $b$)}$$
Note that thus the cluster automorphism which multiplies $(x_1,x_2,y_1,y_2)$ by $\zeta^{\vec w}$, i.e. 
$$\vp:\A\to \A$$ $$x_1\mapsto\zeta x_1,$$ $$x_2\mapsto  x_2,$$ $$y_1\mapsto y_1,$$ $$y_2\mapsto y_2,$$ extends uniquely to a cluster automorphism of $\A$ by Proposition \ref{prop: claut_extension}. Furthermore, we know from the cluster relations at $x_1$ and $x_2$ that this automorphism must act as $x_2\mapsto x_2$ and $x_3\mapsto \zeta^{-1}x_3$.
Now let us examine the effect of this automorphism on points of the form
$$\pt:(x_0,x_1,x_2,x_3,y_1,y_2)\mapsto (\alpha^{-1},0,\alpha,0,\beta_1,\beta_2)$$
Note that $\pt\circ\vp=\pt$, and thus if $b>1$ we can find a cluster automorphism which stabilizes points of the second type, as desired.
\par Noting that if $\pt$ is a deep point of a rank 2 cluster algebra, by Proposition \ref{prop:rk2}, it must be of one of the two types mentioned above, and for both types we can find a cluster automorphism which stabilizes it, we have shown that all rank 2 cluster algebras with the specified frozens have no mysterious points. Utilizing Proposition \ref{prop:RFR}, this shows that all rank 2 cluster algebras have no mysterious points, as desired.

\subsection{Type B}
Note that type $B_n$ is the folding of type $D_{n+1}$. Note to make $B_n$ really full rank, we must add a frozen variable $f$ to either the first node ($n$ odd) or the second node ($n$ even). For example, take the matrices of $B_5$ and $B_6$:
$$B_5: \begin{bmatrix}0&-2&0&0&0\\1&0&-1&0&0\\0&1&0&-1&0\\0&0&1&0&-1\\0&0&0&1&0\\\hline1&0&0&0&0\end{bmatrix}\;\;\;\;\;B_6:\begin{bmatrix}0&-2&0&0&0&0\\1&0&-1&0&0&0\\0&1&0&-1&0&0\\0&0&1&0&-1&0\\0&0&0&1&0&-1\\0&0&0&0&1&0\\\hline0&1&0&0&0&0\end{bmatrix}$$
\begin{proposition}
    Given a deep point $\pt$ of the really full rank type $B$ quiver with frozens as above, there exists a cluster automorphism $\vp$ which fixes $\pt$.
\end{proposition}
The double edge (connecting $x_1$ and $x_2$ in Figure \ref{fig:type_b5b6}) is a separating edge, i.e. if we look at the mutation relation at $x_1$, we have:
$$x_1x_1'=x_2\cdot f+1,$$
for odd $n$ and
$$x_1x_1'=x_2+1,$$
for even $n$.
This means that $p(x_1)$ and $p(x_2)$ can not both simultaneously be 0. We will show that all points with $p(x_1)=0,p(x_1')=0$ which are deep are fixed under a cluster automorphism. Then we will show that all points for which either $p(x_1)\ne0$ or $p(x_1')\ne0$ are not deep by induction.

\begin{figure}[!h]
    \centering
    \begin{tikzpicture}
    
    \begin{scope}[yshift=0cm]
    \node at (-0.7,0) {B$_5$};
    \node[square, label=above :$f$] (B9) at (0,0) {};
    \foreach \x in {1,...,5} \node[node, label=above :$x_\x$] (B\x) at (\x,0) {};
    \draw[-{latex}] (B9)to(B1);
    \foreach \x in {2,...,4} \draw[-{latex}]  (B\the\numexpr\x+1\relax) to (B\x);
    \draw[double, double distance=2pt, postaction={decorate,
            decoration={
                markings,
                mark=at position 0.65 with {\arrow[scale=0.4]{Stealth[]}}
            }
        }] (B2) -- (B1);
\end{scope}

\begin{scope}[yshift=-1cm]
    \node at (-0.7,0) {B$_6$};
    \node[square, label=right :$f$] (B9) at (1,-1) {};
    \foreach \x in {1,...,6} \node[node, label=above :$x_\x$] (B\x) at (\x-1,0) {};
    \draw[-{latex}] (B9) to (B2);
    \foreach \x in {2,...,5} \draw[-{latex}]  (B\the\numexpr\x+1\relax) to (B\x);
    \draw[double, double distance=2pt, postaction={decorate,
            decoration={
                markings,
                mark=at position 0.65 with {\arrow[scale=0.4]{Stealth[]}}
            }
        }] (B2) -- (B1);
\end{scope}
    \end{tikzpicture}
   
    \caption{Example quivers for $B_5$ and $B_6$.}
    \label{fig:type_b5b6}
\end{figure}
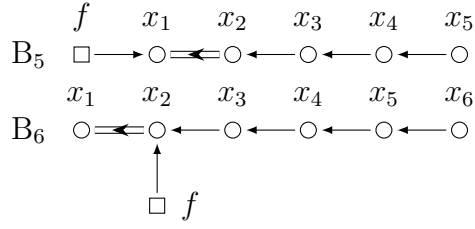

\begin{lemma}
    If $\pt$ is a point of the really full rank type $B$ quiver with frozens as above with $\pt(x_1)=0$ and $\pt(x_1')=0$, then there exists a cluster automorphism $\vp$ fixing $\pt$.
\end{lemma}

\begin{proof} 
Let $\pt$ be any point with $\pt(x_1)=\pt(x_1')=0$. Now, let $\vec w=(1,0,...,0)$, and let $\zeta=-1$. Now note that $\zeta^2=1$, and:
$$B^Tw=\left[
\begin{array}{ccccccc|c}
    0&1&0&0&...&0&0&\mathds{1}_{n\not\in2\Z}
    \\-2&0&1&0&...&0&0&\mathds{1}_{n\in2\Z}\\
    0&-1&0&1&\dots&0&0&0\\
    0&0&-1&0&\dots&0&0&0\\
    \vdots&\vdots&\vdots&\vdots&\ddots&&\vdots&\vdots\\
    0&0&0&0&&0&1&0\\
    0&0&0&0&...&-1&0&0
\end{array}\right]\begin{bmatrix}
    1\\0\\0\\0\\\vdots\\0\\0\\\hline0
\end{bmatrix}=\begin{bmatrix}
    0\\-2\\0\\0\\\vdots\\0\\0
\end{bmatrix}\equiv\vec0\text{ mod 2}$$
This means that the automorphism which negates $x_1$ (and thus $x_1'$) and preserves all other cluster variables in the initial seed and once mutations extends uniquely to an automorphism on the variety. Now, note that any cluster variable can be written as a polynomial in the initial seed and once mutations, as type $B$ cluster algebras are acyclic. Notably, if $\pt(x_1)=\pt(x_1')=0$, the point $\pt$ is preserved under the action of this cluster automorphism. Thus any point $\pt$ with $\pt(x_1)=\pt(x_1')=0$ can not be mysterious. 

\end{proof}
\begin{lemma} Let $\pt$ be a point in the type $B$ cluster variety. Furthermore, assume $\pt(x_1)\ne 0$ or $\pt(x_1')\ne0$. Then $\pt$ is not deep.\end{lemma}
\begin{proof}
    We will proceed by induction on the rank of the cluster algebra. For the base case, notice from Proposition \ref{prop:rk2} that for rank 2, over $\C$ (as in this case, $b=2$, $c=1$ so all deep points are of the second type), any such point can not be deep. This finishes the base case. Now, note that, after freezing at $x_1$ (or $x_1'$), the resulting quiver is isomorphic to type $A_{n-1}$. Now, we note that the deep points of type $A_{n-1}$ are such that $\pt(x_{i})=0$ for all $i$ even (in this naming convention), and notably, if $n-1$ is even, there are no deep points. If $n-1$ is odd, and $\pt$ on $A_{n-1}$ is deep, then $\pt(x_{n-1})\ne0$. This means, in the original quiver, we can freeze at $x_{n-1}$. The resulting quiver will be type $B_{n-2}$ with a disjoint vertex, as $x_{n-1}$ has a leaf to $x_{n}$. Now, we use Lemma \ref{lem:not_deep_leaf} to note that by induction, as type $B_{n-2}$ has no deep points with $\pt(x_1)\ne0$ (or $\pt(x_1')$), that type $B_n$ will have no deep points with $\pt(x_1)\ne0$ (or $\pt(x_1')\ne0$), as desired.
\par Noting that any point with $\pt(x_1)=\pt(x_1')=0$ has a non-trivial automorphism which stabilizes it and all other points are not deep, we have shown that there are no mysterious points in the type $B_n$ cluster variety.
\end{proof}
\subsection{Type C}

Note that type $C_n$ is the folding of type $A_{2n-1}$. We will take a source-sink orientation of the type $C$ quiver. Note to make $C_n$ really full rank, we must add a frozen variable $f$ to the first node regardless of the parity of $n$. For example, take the matrices of $C_5$ and $C_6$:
$$C_5: \begin{bmatrix}0&-1&0&0&0\\2&0&1&0&0\\0&-1&0&-1&0\\0&0&1&0&1\\0&0&0&-1&0\\\hline1&0&0&0&0\end{bmatrix}\;\;\;\;\;C_6:\begin{bmatrix}0&-1&0&0&0&0\\2&0&1&0&0&0\\0&-1&0&-1&0&0\\0&0&1&0&1&0\\0&0&0&-1&0&-1\\0&0&0&0&1&0\\\hline1&0&0&0&0&0\end{bmatrix}$$
We will utilize the naming convention given in Figure \ref{fig:type_C5C6}.
\begin{figure}[!h]
    \centering
    \begin{tikzpicture}
    
    \begin{scope}[yshift=0cm]
    \node at (-0.7,0) {C$_5$};
    \node[square, label=above :$f$] (B9) at (0,0) {};
    \foreach \x in {0,...,4} \node[node, label=above :$x_\x$] (B\x) at (\x+1,0) {};
    \foreach \x in {1,3} \draw[-{latex}]  (B\the\numexpr\x+1\relax) to (B\x) ;
    \draw[-{latex}]  (B9) to (B0) ;
    \draw[-{latex}]  (B2) to (B3);
    \draw[double, double distance=2pt, postaction={decorate,
            decoration={
                markings,
                mark=at position 0.65 with {\arrow[scale=0.4]{Stealth[]}}
            }
        }] (B0) -- (B1);
\end{scope}

\begin{scope}[yshift=-1cm]
    \node at (-0.7,0) {C$_6$};
    \node[square, label=above :$f$] (B9) at (0,0) {};
    \foreach \x in {0,...,5} \node[node, label=above :$x_\x$] (B\x) at (\x+1,0) {};
    \foreach \x in {2,4} \draw[-{latex}] (B\x)to (B\the\numexpr\x+1\relax) ;
    \draw[-{latex}] (B9) to (B0);
    \foreach \x in {2,4} \draw[-{latex}] (B\x)to (B\the\numexpr\x-1\relax) ;
    \draw[double, double distance=2pt, postaction={decorate,
            decoration={
                markings,
                mark=at position 0.65 with {\arrow[scale=0.4]{Stealth[]}}
            }
        }] (B0) -- (B1);
\end{scope}
    \end{tikzpicture}
   
    \caption{Type $C_5$ and $C_6$ quivers in source-sink orientation. We call the double node $x_0$.}
    \label{fig:type_C5C6}
\end{figure}
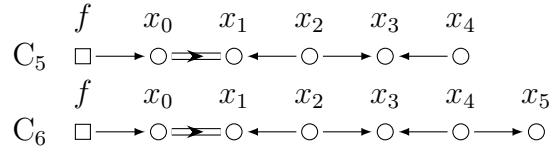
\par We will show that type $C$ cluster varieties have no mysterious points with a series of lemmas.
We will show that the points $\pt(x_i)=\pt(x_i')=0$ for $i$ even are stabilized by a cluster automorphism when $n$ odd, the points where $\pt(x_i)=\pt(x_i')=0$ for $i$ odd are stabilized by a cluster automorphism when $n$ even, and all other points are not deep.
\begin{lemma}
    If $n$ is odd, and $\pt$ is a point in the type $C_n$ cluster variety with frozens as above with $\pt(x_i)=\pt(x_i')=0$ for all $i$ even, then there exists $\vp$ such that $\vp(\pt)=\pt$.
    \label{lem:C_n_odd}
\end{lemma}
\begin{proof}
    As in the lemma statement, assume that $\pt(x_i)=\pt(x_i')=0$ for $i$ even and $n$ odd. Take $\vec w=(1,0,1,0\dots,0,1)$, and $\zeta=-1$. We can calculate the following.
$$B^Tw=\left[
\begin{array}{ccccccc|c}
    0&2&0&0&...&0&0&1\\
    -1&0&-1&0&...&0&0&0\\
    0&1&0&1&...&0&0&0\\
    0&0&-1&0&...&0&0&0\\
    \vdots&\vdots&\vdots&\vdots&\ddots&&\vdots&\vdots\\0&0&0&0&&0&-1&0
    \\
    0&0&0&0&...&1&0&0
\end{array}\right]\begin{bmatrix}
    1\\0\\1\\0\\\vdots\\0\\1\\\hline 0
\end{bmatrix}=\begin{bmatrix}
    0\\-2\\0\\-2\\\vdots\\-2\\0
\end{bmatrix}\equiv\vec0\text{ mod 2}$$
This means that the cluster automorphism which negates the even cluster variables and their once mutations extends uniquely to a cluster automorphism on $C_n$ when $n$ odd, and thus any point with $\pt(x_i)=\pt(x_i')=0$ for $i$ even, has a non-trivial cluster automorphism which fixes it.
\end{proof}
\begin{lemma}
    If $n$ is even, and $\pt$ is a point in the type $C_n$ cluster variety with frozens as above with $\pt(x_i)=\pt(x_i')=0$ for all $i$ odd, then there exists $\vp$ such that $\vp(\pt)=\pt$.
    \label{lem:C_n_even}
\end{lemma}
\begin{proof} 
Assume, as in the problem statement, that $\pt(x_i)=\pt(x_i')=0$ for $i$ odd and $n$ even.  Take $\vec w=(0,1,0,1\dots,0,1)$and $\zeta=-1$. We can calculate the following.
$$B^Tw=\left[
\begin{array}{ccccccc|c}
    0&2&0&0&...&0&0&1\\
    -1&0&-1&0&...&0&0&0\\
    0&1&0&1&...&0&0&0\\
    0&0&-1&0&...&0&0&0\\
    \vdots&\vdots&\vdots&\vdots&\ddots&&\vdots&\vdots\\0&0&0&0&&0&1&0
    \\
    0&0&0&0&...&-1&0&0
\end{array}\right]\begin{bmatrix}
    0\\1\\0\\1\\\vdots\\0\\1\\\hline 0
\end{bmatrix}=\begin{bmatrix}
    2\\0\\2\\0\\\vdots\\2\\0\end{bmatrix}\equiv\vec0\text{ mod 2}$$

This means that the cluster automorphism which negates the odd cluster variables and their once mutations extends uniquely to a cluster automorphism on $C_n$ when $n$ even, and thus any point with $\pt(x_i)=\pt(x_i')=0$ for $i$ odd, has a non-trivial cluster automorphism which fixes it.
\end{proof}
\begin{lemma}
    Points in the type $C$ cluster variety, with frozens as above, not of the types in Lemmas \ref{lem:C_n_odd} and \ref{lem:C_n_even} are not deep.
\end{lemma}
\begin{proof}
 We will proceed by induction. Note that the claim holds when $n=2$ by Proposition \ref{prop:rk2} (Note the label scheme is flipped, as we label the initial seed as $(x_0,x_1)$ instead of $(x_1,x_2)$, and we have $c=2$, $b=1$.). 
 \\
 \par \textbf{Case 1: $n$ is odd.} Suppose at least one of $\pt(x_i)$ or $\pt(x_i')$ is non-zero for $i$ even. Taking the source-sink orientation of the type $C$ quiver ensures that after one mutation, the unoriented diagram will remain unchanged if we mutate at $x_i$. Now, we note that if $\pt(x_i)\ne0$, then $\pt$ lies in the image of a cluster chart given by the cluster algebra obtained by freezing at $i$ (or $i'$ if $\pt(x_i')\ne0$) by \cite[Lemma 3.2]{Muller2013LocallyAcyclic}. We can thus examine the cluster algebra with $x_i$ (or respectively, $x_i'$) frozen.
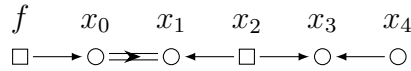
\begin{figure}[!h]
    \centering
    \begin{tikzpicture}
    
    \begin{scope}[yshift=0cm]
    \node[square, label=above :$f$] (B9) at (0,0) {};
    \foreach \x in {0,1,3,4} \node[node, label=above :$x_\x$] (B\x) at (\x+1,0) {};
    \draw[-{latex}](B9) to (B0);
    \node[square,label=above :$x_2$] (B2) at (2+1,0) {};
    \foreach \x in {1,3} \draw[-{latex}]  (B\the\numexpr\x+1\relax) to (B\x) ;
    \draw[-{latex}]  (B2) to (B3);
    \draw[double, double distance=2pt, postaction={decorate,
            decoration={
                markings,
                mark=at position 0.65 with {\arrow[scale=0.4]{Stealth[]}}
            }
        }] (B0) -- (B1);
\end{scope}

    \end{tikzpicture}
   
    \caption{An example of freezing at $x_i$ for $i$ even, $i>0$. Notably, the left part of the diagram is isomorphic to our type $C_2$ diagram, and the right part is isomorphic to the type $A_2$ diagram.}
    \label{fig:type_C5f}
\end{figure}
\par Now, note that freezing at this vertex results in the union of a type $C_{k}$ and $A_{l}$ quiver over a frozen where $k+l+1=n$. Notably, $k$ and $l$ are also both even (as $i$ was an even index) (see Figure \ref{fig:type_C5f}). 

\par Note that the clusters of this cluster algebra are a subset of the clusters of the $C_n$ cluster algebra. This means that if $\pt$ is a deep point of type $C_n$ which is in this cluster chart, then $\pt$ is deep in the cluster algebra associated to this chart. Furthermore, assume that $\pt|_{C_k}$ and $\pt|_{A_l}$ are both not deep. Then we can find a cluster for each that has all cluster variables non zero. Then we can use the same sequence of mutations on the $C_n$ quiver to obtain a cluster where $\pt$ evaluates to non zero on all cluster variables. Thus at least one of $\pt|_{C_k}$ or $\pt|_{A_l}$ must be deep. This analysis applies to any such $i$ with $\pt(x_i)$ or $\pt(x_i')\ne0$.

\par Now, take the minimum even $i$ that you can do this with. If $i=0$, then the resulting quiver is type $A_{n-1}$ which has no deep points, and we are done. For any other $i$, the resulting quiver will be $C_k$ with $k>0$ and even. By the inductive hypothesis, the only deep points of $C_k$ with $k$ even have $\pt(x_0)\ne 0$ (the inductive hypothesis gives $\pt(x_1)=0=\pt(x_1')$, and the first edge is a separating edge). This is a contradiction of either $\pt$ being deep, or $i$ being the smallest freezable value. Therefore, $\pt$ restricted to this $C_k$ component can not be deep. But $A_l$ has no deep points for $l$ even. Thus, after freezing at $i$ the point is no longer deep, and thus it was not a deep point.\\
 \par \textbf{Case 2: $n$ is even.} Let $\pt$ be a point with $\pt(x_i)\ne 0$ or $\pt(x_i')\ne 0$ for some odd $i$. Let $j$ be the smallest such $i$ and we will investigate freezing at $j$. If $j=1$, note that we can use the same argument as Lemma \ref{lem:not_deep_leaf}, noting that $C_2$ has no deep points with $\pt(x_1)\ne0$, and type $A_{n-2}$ has no deep points, we have that the point can not be deep (i.e. freeze at $x_1$, mutate the $A_{n-2}$ until the point is not deep, freeze at all vertices in $A_{n-2}$ then thaw $x_1$, and the resulting quiver has no deep points with $\pt(x_1)\ne0$).  
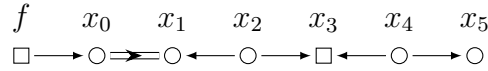
\begin{figure}[!h]
    \centering
    \begin{tikzpicture}

\begin{scope}[yshift=-1cm]
    \node[square, label=above :$f$] (B9) at (0,0) {};
    \foreach \x in {0,1,2,4,5} \node[node, label=above :$x_\x$] (B\x) at (\x+1,0) {};
    \draw[-{latex}](B9) to (B0);
    \node[square, label=above :$x_3$] (B3) at (3+1,0) {};
    \foreach \x in {2,4} \draw[-{latex}] (B\x)to (B\the\numexpr\x+1\relax) ;
    \foreach \x in {2,4} \draw[-{latex}] (B\x)to (B\the\numexpr\x-1\relax) ;
    \draw[double, double distance=2pt, postaction={decorate,
            decoration={
                markings,
                mark=at position 0.65 with {\arrow[scale=0.4]{Stealth[]}}
            }
        }] (B0) -- (B1);
\end{scope}
    \end{tikzpicture}
   
    \caption{An example of freezing at $x_i$ for $i$ odd, $i>1$. Notably, the left part of the diagram is the same as the diagram for $C_3$ and the right part is the diagram for $A_2$}
    \label{fig:type_C6f}
\end{figure}
\par For all other $j$, note that the resulting quiver is $C_k$ and $A_l$ with $k$ odd and $l$ even, see Figure \ref{fig:type_C6f}. Now, note that, by the fact that $j$ was minimal, $\pt(x_i)=\pt(x_i')=0$ for all odd $i<j$, and $k<n$. Now we note by the inductive hypothesis, such a point is not deep. (as $\pt(x_i)=\pt(x_i')=0$ for all even $i$ in a deep point, and the edges are all separating, as the quiver is acyclic). Furthermore, we note that type $A_l$ has no deep points for $l$ even. Thus, after freezing at $j$, the resulting point is not deep. Thus if $n$ even and $\pt(x_i)$ or $\pt(x_i')$ is non-zero, the point is not deep. Thus the only possible deep points are of the types specified in Lemmas \ref{lem:C_n_odd} and \ref{lem:C_n_even}.
\end{proof}
\subsection{Type \texorpdfstring{$F_4$}{F4}}

First, note that the Dynkin diagram of type $F_4$ is really full rank,

$$F_4:\begin{bmatrix}
    0&-1&0&0\\1&0&-2&0\\0&1&0&-1\\0&0&1&0
\end{bmatrix}$$

We can ignore any potential frozens. We will use the following labels on the initial quiver.
\begin{figure}[!h]
    \centering
    \begin{tikzpicture}
    
    \begin{scope}[yshift=0cm]
    \node at (0.3,0) {F$_4$};
    \foreach \x in {1,...,4} \node[node, label=above :$x_\x$] (B\x) at (\x,0) {};
    \draw[-{latex}](B2) to (B1);
    \draw[-{latex}](B4) to (B3);
    \draw[double, double distance=2pt, postaction={decorate,
            decoration={
                markings,
                mark=at position 0.65 with {\arrow[scale=0.4]{Stealth[]}}
            }
        }] (B3) -- (B2);
\end{scope}

    \end{tikzpicture}
   
    \caption{The quiver for the type $F_4$ cluster algebra. The orientation above is chosen for convenience.}
    \label{fig:type_F4}
\end{figure}
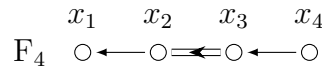
Now, we note that by Lemma \ref{lem:not_deep_leaf}, as the first and last vertices are leaves, and removing the second or third, the remaining quiver is type $A_2$, which has no deep points. Thus, if $\pt$ were a deep point, it must not take non-zero values on both vertex two and vertex three in the initial cluster, i.e. $\pt(x_2)=\pt(x_3)=0$. But, we know that the mutation relation at vertex two tells us
$$x_2x_2'=x_3^2+x_1$$
and thus the points where $\pt(x_2)=\pt(x_3)=0$ must also have $\pt(x_1)=0$. But note that the cluster relation at $x_1$ tells us 
$$x_1x_1'=x_2+1$$
But this implies that there is no point with $\pt(x_1)=\pt(x_2)=0$,  which contradicts what a deep point in this variety must have. Thus no point in the type $F_4$ cluster variety can be deep, and as there are no deep points, there are no mysterious points.

\bibliographystyle{alpha}
\bibliography{surface_deep_pts}

\end{document}